\documentclass[twoside,12pt]{article}
\usepackage{amsmath,amssymb,graphicx,amsthm}
\usepackage[latin1]{inputenc}
\usepackage{anysize}
\marginsize{2cm}{2cm}{2cm}{3cm}

\usepackage{hyperref}
\hypersetup{colorlinks=true, citecolor=blue} 

\usepackage{tikz}
\usepackage{colortbl}
\usepackage{stmaryrd} 

\theoremstyle{definition}
  \newtheorem{definition}{Definition}[section]

\theoremstyle{plain}
  \newtheorem{theorem}{Theorem}[section]
  
  \newtheorem{corollary}{Corollary}[section]
  \newtheorem{lemma}{Lemma}[section]

\usepackage{dsfont}
\usepackage{multirow}

\usepackage{caption}
\usepackage{subcaption}

\usepackage{lineno} 
\title{Delta shock solution to the generalized one-dimensional zero-pressure gas dynamics system with linear damping}
\author{Richard De la cruz \thanks{Universidad Pedagógica y Tecnológica de Colombia, School of Mathematics and Statistics, 150003, Colombia. {\em e-mail:} \href{mailto:richard.delacruz@uptc.edu.co}{richard.delacruz@uptc.edu.co} } 
\\ Juan Carlos Juajibioy \thanks{Universidad Pedagógica y Tecnológica de Colombia, School of Mathematics and Statistics, 150003, Colombia. {\em e-mail:} \href{mailto:juan.juajibioy@uptc.edu.co}{juan.juajibioy@uptc.edu.co}}
}
\date{ }

\begin{document}
\maketitle

\begin{abstract}
In this paper, we propose a time-dependent viscous system and by using  the vanishing viscosity method we show the existence of delta shock solution for a generalized one-dimensional zero-pressure gas dynamics system with linear damping. 
\end{abstract}
\section{Introduction}
In this paper, we study the Riemann problem to the following hyperbolic system of conservation laws with linear damping
\begin{equation}
\label{system_ld}
  \begin{cases}
   v_t+\left(vu^k \right)_x=0,\\
   (vu)_t+(vu^{k+1})_x=-\alpha vu,
  \end{cases}
 \end{equation}
where $k$ is an odd natural number, $\alpha>0$ is a constant and the initial data is given by
\begin{equation} \label{datoRiemann}
 (v(x,0),u(x,0))=\begin{cases}
                (v_-,u_-), &\text{if } x<0,\\
                (v_+,u_+), &\text{if } x>0,
               \end{cases}
\end{equation}
for arbitrary constant states $(v_\pm,u_\pm)$ with $v_\pm>0$. 
It is well known that the system \eqref{system_ld} is not strictly hyperbolic with eigenvalue $\lambda = u^k$ and right eigenvector ${\bf r}=(1,0)$. Moreover, $\nabla \lambda \cdot {\bf r}=0$ and therefore the system is linearly degenerate.
The principal reason to choose the condition on $k$ is due to physical motivation. In fact, for $k=1$ the system \eqref{system_ld} becomes the one-dimensional zero-pressure gas dynamics system with linear damping.
In the one-dimensional zero-pressure gas dynamics 
$v\ge 0$ denotes the density of mass and $u$ the velocity. The one-dimensional zero-pressure gas dynamics system can be used to describe the motion process of free particles sticking under collision in the low temperature and the information of large-scale structures in the universe \cite{BG, SZ}.
Really, the one-dimensional zero-pressure gas dynamics system arise in a wide variety of models in physics, for a more detailed description see for example \cite{Bouchut, ERS, MMZ, LeVeque}. 
So, more details on the studies of the one-dimensional zero-pressure gas dynamics system can be found in \cite{Bouchut, BG, ERS, HW, LeVeque, LY}.
 The Riemann problem for the following system 
\begin{equation*} 
\begin{cases}
v_t+(vf(u))_x=0,\\
(vu)_t+(vuf(u))_x=0.
\end{cases}
\end{equation*}
was solved completely in \cite{HYang} by mean of characteristic analysis and the vanishing viscosity method.
 In 2016, Shen \cite{Shen} studied the Riemann problem for the one-dimensional zero-pressure gas dynamics with Coulomb-like friction term and the solutions involve delta shock wave and vacuum state. Shen's paper is the first work for the one-dimensional zero-pressure gas dynamics system with a source term.
 Recently, Keita and Bourgault \cite{KB} solved the Riemann problem to the one-dimensional zero-pressure gas dynamics system with linear damping (i.e. the system \eqref{system_ld} when $k=1$) and their results include delta shock wave solution.\\
 
  In this paper, 
we will find that the delta shock wave solution appears in the Riemann solution to the system \eqref{system_ld} with inital data \eqref{datoRiemann}.   Therefore, we
propose the following time-dependent viscous system
 \begin{equation}
\label{system_ld_eps}
  \begin{cases}
   v_t+\left(vu^k \right)_x=0,\\
   (vu)_t+(vu^{k+1})_x=\varepsilon \frac{1}{\alpha k} e^{-\alpha kt}(1-e^{-\alpha kt}) u_{xx}-\alpha vu,
  \end{cases}
 \end{equation}
where  $k$ is an odd natural number and  $\alpha>0$ is a constant.
Notice that when $\alpha \to 0+$, we have that $\lim\limits_{\alpha \to 0+} \frac{1}{\alpha k} e^{-\alpha kt}(1-e^{-\alpha kt})=t$ and the system \eqref{system_ld_eps} coincides with the viscous system (4.1) in \cite{HYang} with $f(u)=u^k$.
The viscous system \eqref{system_ld_eps} is well motivated by scalar conservation law with time-dependent viscosity
\begin{equation*} \label{GBurgers}
u_t+ F(u)_x = G(t)u_{xx}.
\end{equation*}
where $G(t) > 0$ for $t > 0$.
When $F(u)=u^2$ the scalar equation is called the Burgers equation with time-dependent viscosity. This type of Burgers equation
was studied as a mathematical model of the propagation of the finite-amplitude sound waves in variable-area ducts, where $u$ is an acoustic variable, with the linear effects of changes in the duct area taken out, and the time-dependent viscosity $G(t)$ is  the  duct  area, see \cite{Crighton, DE, WZ}.
The reader can find results concerning the existence, uniqueness and explicit solutions to the Burgers equation with time-dependent viscosity with suitable conditions for $G(t)$ in \cite{Crighton, CS, DE, Scott, WZ2, WZ, HZhang, ZW} and references cited therein.
The Burgers equation with time-dependent viscosity and linear damping
was studied in \cite{MVSK} and their results include explicit solutions for differents $G(t)$.\\
When $G(t)=\varepsilon t$ and $\varepsilon>0$, for systems of hyperbolic conservation laws with time-dependent viscosity we refered the works developed by Tupciev in \cite{Tupciev} and Dafermos in \cite{Dafermos}. The results obtained in \cite{Dafermos} and \cite{Tupciev} not including the delta shock waves solutions.
For systems of hyperbolic conservation laws with delta shock solutions
the reader may consult \cite{DS, Ercole, TZZ, HYang, YZ}. Recently, in \cite{Delacruz} a nonlinear time-dependent viscosity $G(t)$ was consider to obtained delta shock wave solution for a particular system of conservation laws with linear damping.\\

Notice that our proposal of the time-dependent viscous system \eqref{system_ld_eps} is a special case of the general systems of conservation laws with time-dependent viscous system.
Therefore, we consider the viscous system \eqref{system_ld_eps} with initial data \eqref{datoRiemann}.
 Observe that if $(\widehat{v},\widehat{u})$ solves
\begin{equation} \label{sysVis1}
\begin{cases}
\widehat{v}_t+e^{-\alpha kt}(\widehat{v} \widehat{u}^k)_x=0,\\
(\widehat{v} \widehat{u})_t+ e^{-\alpha kt} \left( \widehat{v} \widehat{u}^{k+1} \right)_x  =  \varepsilon \frac{1}{\alpha k}e^{-\alpha kt}(1-e^{-\alpha kt})\widehat{u}_{xx},
\end{cases}
\end{equation} 
with initial condition
\begin{equation} \label{datsysVisc1}
 (\widehat{v}(x,0),\widehat{u}(x,0))=\begin{cases}
                (v_-,u_-), &\text{if } x<0,\\
                (v_+,u_+), &\text{if } x>0,
               \end{cases}
\end{equation}
then $(v,u)$ defined by $(v,u)=(\widehat{v},\widehat{u}e^{-\alpha t})$ solves the problem \eqref{system_ld_eps}--\eqref{datoRiemann}. 
Here and after,  denote $(\widehat{v}^\varepsilon, \widehat{u}^\varepsilon)$ as $(\widehat{v}, \widehat{u})$ when there is no confusion.
In order to solve the problem \eqref{sysVis1}--\eqref{datsysVisc1}, we introduce the similarity variable $\xi$ and solutions to \eqref{sysVis1} should approach for large times a similarity 
solution $(\widehat{v},\widehat{u})$ to \eqref{sysVis1} of the form $\widehat{v}(x,t)=\widehat{v}(\xi)$, $\widehat{u}(x,t)=\widehat{u}(\xi)$ and $\xi= a(t) x$ for some suitable smooth function $a(t) \ge 0$ for $t>0$ (more details on the similarity methods can be found in \cite{Barenblatt, Henriksen, Luo, Polyanin, Sachdev, Suto} and references therein).
Therefore, we introduce the similarity variable $\xi=\frac{\alpha kx}{1-e^{-\alpha kt}}$ and the system \eqref{sysVis1} can be written as
\begin{equation} \label{sysVis2}
\begin{cases} 
-\xi \widehat{v}_\xi+(\widehat{v}\widehat{u}^k)_\xi=0,\\
- \xi  (\widehat{v}\widehat{u})_\xi+(\widehat{v}\widehat{u}^{k+1})_\xi = \varepsilon \widehat{u}_{\xi \xi},
\end{cases}
\end{equation} 
and the initial data becomes the boundary condition
\begin{equation} \label{datsysVisc2}
(\widehat{v}(\pm \infty), \widehat{u}(\pm \infty))=(v_\pm,u_\pm).
\end{equation}
Note that when $\alpha \to 0+$, the similarity variable $\xi$ converges to $x/t$ which is well used in many methods to study the behavior and structure of solutions of nonlinear hyperbolic systems of conservation laws.
Now, when $\varepsilon \to 0+$, the system \eqref{sysVis1} becomes
\begin{equation} \label{sys1}
\begin{cases}
\widehat{v}_t+e^{-\alpha kt}(\widehat{v} \widehat{u}^k)_x=0,\\
(\widehat{v} \widehat{u})_t+ e^{-\alpha kt} \left( \widehat{v} \widehat{u}^{k+1} \right)_x  =  0.
\end{cases}
\end{equation} 
Using the vanishing viscosity method, and following works by Tan, Zhang and Zheng \cite{TZZ}, Li and Yang \cite{LY}, Yang \cite{HYang} and De la cruz and Santos \cite{DS} with some appropriate modifications, 
we show the existence of solutions for system \eqref{sysVis2} with boundary condition \eqref{datsysVisc2}.
 The main difficulty in applying the vanishing viscosity method developed in \cite{LY, TZZ,  HYang} is to choose a suitable Banach space and a bounded convex closed subset to use the Schauder fixed point theorem. Therefore, we specifically follow the vanishing viscosity method developed in \cite{DS}.
After, we study the behavior of the solutions $(\widehat{v}^\varepsilon, \widehat{u}^\varepsilon)$ as $\varepsilon \to 0+$ to obtain a delta shock wave solution for the system \eqref{sys1}. Finally, as $(v(x, t), u(x, t)) = (\widehat{v}(x,t), \widehat{u}(x,t)e^{-\alpha t})$, the solutions of \eqref{sys1} are using to obtain solutions for the original system \eqref{system_ld}.
\\

The outline of the remaining of the paper is as follows. In Section 2, we show the existence of solutions to the viscous system \eqref{sysVis2} with boundary condition \eqref{datsysVisc2}. In Section 3, we study the behavior of the solutions $(\widehat{v}^\varepsilon, \widehat{u}^\varepsilon)$ as $\varepsilon \to 0+$ and we show the existence of delta shock solution for the system \eqref{sysVis1} without viscosity. In Section 4, we show the existence of delta shock solution for the nonhomogeneous system \eqref{system_ld}. Final remarks are given in Section 5.

 \section{Existence of solutions to the viscous system \texorpdfstring{\eqref{sysVis2}-\eqref{datsysVisc2}}{(7)-(8)} }

Let $R$ be a positive number such that $R^{1/k} > \max \{ |u_-|,|u_+| \}$. We consider the Banach space $C([-R,R])$, endowed with the supremum norm, and 
let us define the following set
\begin{small}
 $$ K = \{ U \in C([-R,R]) \, | \, U \mbox{ is monotone increasing with } U(-R) = u_- \mbox{ and } U(R) = u_+ \} $$
\end{small}
which is bounded and a convex closed set in $C([-R,R])$.
\begin{lemma}
Suppose 
$U \in K \cap C^1([-R,R])$. 
 Let 
 \begin{equation} \label{solweak1}
  \widehat{v}(\xi)=\begin{cases}
             \widehat{v}_1(\xi), &\mbox{if }-R\le \xi < \xi_\sigma^\varepsilon,\\
          \widehat{v}_2(\xi), &\mbox{if }\xi_\sigma^\varepsilon < \xi \le R,
            \end{cases}
 \end{equation}
where $\xi_\sigma^\varepsilon$ is the unique solution of the equation $(U(\xi))^{k}=\xi$  (which solution exists because $u_->u_+$ and $R$ is big enough), 
 \begin{equation} \label{solRho1}
  \widehat{v}_1(\xi):=v_-\frac{u_-^k+R}{(U(\xi))^k-\xi}\exp \left(-\int_{-R}^\xi \frac{ds}{(U(s))^k-s} \right)
 \end{equation}
 and 
 \begin{equation} \label{solRho2}
  \widehat{v}_2(\xi):=v_+\frac{R-u_+^k}{\xi-(U(\xi))^k}\exp \left(\int_\xi^R \frac{ds}{(U(s))^k-s} \right).
 \end{equation}
 Then $\widehat{v} \in L^1([-R,R])$, $v$ is continuous in $[-R,\xi_\sigma^\varepsilon) \cup (\xi_\sigma^\varepsilon,R]$ and it is a weak solution for 
 \begin{equation} \label{eq1}
	   -\xi \widehat{v}_\xi+(\widehat{v} U^k)_\xi=0,
	\end{equation} 
and $\widehat{v}(\pm R)=v_\pm$.
\end{lemma} 
\begin{proof}
 The equation \eqref{eq1} can be rewritten as
 \begin{equation} \label{eq4}
  ((U(\xi))^k-\xi)\widehat{v}'+\widehat{v}((U(\xi))^k)'=0.
 \end{equation}
 Integrating \eqref{eq4} on $[-R,\xi]$ for $-R<\xi<\xi_\sigma^\varepsilon$ , we get
 \begin{equation} \label{eq5}
  ((U(\xi))^k-\xi)\widehat{v}_1(\xi)-(u_-^k+R)v_-+  \int_{-R}^\xi \widehat{v}_1(s)ds=0.
 \end{equation}
 Let
 $$ p(\xi)=\int_{-R}^\xi \widehat{v}_1(s)ds, \quad A_1=(u_-^k+R)v_- \quad \mbox{ and } a(\xi)=((U(\xi))^k-\xi). $$
 Then \eqref{eq5} can be written as
 \begin{equation*}
  \begin{cases}
   a(\xi)p'(\xi)+p(\xi)=A_1,\\
   p(-R)=0.
  \end{cases}
 \end{equation*}
It follows that
$$ p(\xi)=A_1 \left\{ 1-\exp \left( -\int_{-R}^\xi \frac{ds}{a(s)} \right) \right\}. $$
Noticing that $a(\xi)>0$ and $a(\xi)=O(|\xi-\xi_\sigma|)$ as $\xi \to \xi_\sigma^\varepsilon-$, we obtain
\begin{equation} \label{eqA}
 \lim_{\xi \to \xi_\sigma^\varepsilon-} \int_{-R}^\xi \widehat{v}_1(s) ds = \lim_{\xi \to \xi_\sigma^\varepsilon-} p(\xi)=A_1. 
\end{equation}
Hence
\begin{equation} \label{eqEst1}
 \lim_{\xi \to \xi_\sigma^\varepsilon-} ((U(\xi))^k-\xi)\widehat{v}_1(\xi)=0.
\end{equation}
Similarly, one can get
\begin{align}
 \lim_{\xi \to \xi_\sigma^\varepsilon+} \int_{\xi}^R \widehat{v}_2(s) ds =A_2, \label{eqB}\\
 \lim_{\xi \to \xi_\sigma^\varepsilon+} ((U(\xi))^k-\xi)\widehat{v}_2(\xi)=0, \nonumber
\end{align}
where $A_2=(u_+^k-R)v_+$. The equalities \eqref{eqA} and \eqref{eqB} imply that $\widehat{v}(\xi) \in L^1([-R, R])$.\\

Now, for arbitrary $\phi \in C_0^\infty ([-R,R])$, we verify that
\begin{equation*}
 I \equiv -\int_{-R}^R ((U(\xi))^k-\xi)\widehat{v}(\xi) \phi'(\xi) d\xi +  \int_{-R}^R \widehat{v}(\xi)\phi(\xi) d\xi=0.
\end{equation*}
For any $\xi_1, \xi_2$ , such that $-R<\xi_1 <\xi_\sigma^\varepsilon <\xi_2 <R$  we can write $I=I_1+I_2+I_3$, where 
\begin{align*}
 I_1&= \int_{-R}^{\xi_1} (-((U(\xi))^k-\xi)\widehat{v}(\xi) \phi'(\xi)+\widehat{v}(\xi)\phi(\xi)) d\xi,\\
 I_2&= \int_{\xi_1}^{\xi_2} (-((U(\xi))^k-\xi)\widehat{v}(\xi) \phi'(\xi)+\widehat{v}(\xi)\phi(\xi)) d\xi \mbox{ and }\\
 I_3&= \int_{\xi_2}^{R} (-((U(\xi))^k-\xi)\widehat{v}(\xi) \phi'(\xi)+\widehat{v}(\xi)\phi(\xi)) d\xi.
\end{align*}
Observe that 
\begin{align*}
 |I_1|&=\left| -((U(\xi_1))^k-\xi_1)\widehat{v}_1(\xi_1)\phi(\xi_1)+\int_{-R}^{\xi_1} ((((U(\xi))^k-\xi)\widehat{v}(\xi))' \phi(\xi)+\widehat{v}(\xi)\phi(\xi))) d\xi \right|\\
 &=\left| ((U(\xi_1))^k-\xi_1)\widehat{v}_1(\xi_1)\phi(\xi_1) \right|.
\end{align*}
By \eqref{eqEst1}, we have that $$\lim_{\xi_1 \to \xi_\sigma^\varepsilon-} |I_1|=\lim_{\xi_1 \to \xi_\sigma^\varepsilon-} \left| ((U(\xi_1))^k-\xi_1)\widehat{v}_1(\xi_1)\phi(\xi_1) \right|=0.$$
Similarlly, we show that 
$$\lim_{\xi_2 \to \xi_\sigma^\varepsilon+} |I_3|=\lim_{\xi_2 \to \xi_\sigma^\varepsilon+} \left| ((U(\xi_2))^k-\xi_2)\widehat{v}_2(\xi_2)\phi(\xi_2) \right|=0.$$
Since $\widehat{v} \in L^1([-R,R])$, 
$$ |I_2| \le \int_{\xi_1}^{\xi_2} |-((U(\xi))^k-\xi)\phi'(\xi)+\phi(\xi)||\widehat{v}(\xi)|d\xi \to 0, \quad \mbox{ as } \xi_1\to \xi_\sigma^\varepsilon-, \xi_2\to \xi_\sigma^\varepsilon+. $$
But $I$ is independent of $\xi_1$ and $\xi_2$, so $I=0$. Therefore, $\widehat{v}$ defined in \eqref{solweak1} is a weak solution.
\end{proof}

Now, let us define an operator $T: K \to C^2([-R,R])$ as follows: for any $U \in K$, $\widehat{u}=TU$ is the unique solution of the boundary value problem
    \begin{equation} \label{PRegA}
     \begin{cases}
      \varepsilon \widehat{u}''=\left( \widehat{v}(U,\xi)((U(\xi))^k-\xi) \right) \widehat{u}',\\
      \widehat{u}(\pm R)=u_\pm,
     \end{cases}
    \end{equation}
where $\widehat{v}(U,\xi) \equiv \widehat{v}(\xi)$ is defined in \eqref{solRho1} or \eqref{solRho2}.
 In fact, the solution to this problem can be found explicitly and it is given by
\begin{equation} \label{Sol1u}
 \widehat{u}(\xi)=u_- + \frac{(u_+-u_-)\int_{-R}^\xi \exp \left( \int_{-R}^r \frac{\widehat{v}(U,s)((U(s))^k-s)}{\varepsilon} ds \right)dr}{\int_{-R}^R \exp \left( \int_{-R}^r \frac{\widehat{v}(U,s)((U(s))^k-s)}{\varepsilon} ds \right)dr}.
\end{equation}

\begin{lemma}
  $T:K \to K$ is a continuous operator.
 \end{lemma}
\begin{proof}
 Choose $\{ U_n\}$ in $K$ such that $U_n \to U$. As $U$ belongs to $K$, then each $\widehat{u}_n=TU_n$ and $\widehat{u}=TU$ satisfy the problem \eqref{PRegA}. Now, we have the following problem
 \begin{small}
  \begin{equation} \label{PRegB}
  \begin{cases}
   \varepsilon (\widehat{u}_n-\widehat{u})''=(\widehat{v}(U_n,\xi)((U_n(\xi))^k-\xi))(\widehat{u}_n-\widehat{u})' \\
  \qquad \qquad \qquad +(\widehat{v}(U_n,\xi)((U_n(\xi))^k-\xi) - \widehat{v}(U,\xi)((U(\xi))^k-\xi))\widehat{u}'\\
   (\widehat{u}_n-\widehat{u})(\pm R)=0.
  \end{cases}
 \end{equation}
 \end{small}
 Setting $p_n(\xi)=\widehat{v}(U_n,\xi)((U_n(\xi))^k-\xi)$ and $q_n(\xi)=(\widehat{v}(U_n,\xi)((U_n(\xi))^k-\xi) - \widehat{v}(U,\xi)((U(\xi))^k-\xi))u'$, from problem \eqref{PRegB} we have
 \begin{align}
  (\widehat{u}_n-\widehat{u})'(\xi)=&-\frac{\int_{-R}^R \int_{-R}^y \frac{q_n(r)}{\varepsilon} \exp \left( \int_r^y \frac{p_n(s)}{\varepsilon} ds \right) dr dy}{\int_{-R}^R \exp \left( \int_{-R}^r \frac{p_n(s)}{\varepsilon}ds \right)dr} \exp \left( \int_{-R}^\xi \frac{p_n(s)}{\varepsilon} ds \right) \nonumber \\
  &+\int_{-R}^\xi \frac{q_n(r)}{\varepsilon} \exp \left( \int_{-r}^\xi \frac{p_n(s)}{\varepsilon} ds \right) dr \label{eqSeq1}
 \end{align}
 and
 \begin{small}
 \begin{align}
  (\widehat{u}_n-\widehat{u})(\xi)=&-\frac{\int_{-R}^R \int_{-R}^y \frac{q_n(r)}{\varepsilon} \exp \left( \int_r^y \frac{p_n(s)}{\varepsilon} ds \right) dr dy}{\int_{-R}^R \exp \left( \int_{-R}^r \frac{p_n(s)}{\varepsilon}ds \right)dr} \int_{-R}^\xi \exp \left( \int_{-R}^r \frac{p_n(s)}{\varepsilon} ds \right)dr \nonumber \\
  &+\int_{-R}^\xi \int_{-R}^y \frac{q_n(r)}{\varepsilon} \exp \left( \int_{-r}^y \frac{p_n(s)}{\varepsilon} ds \right) dr dy. \label{eqSeq2}
 \end{align}
 \end{small}
 From \eqref{eq1}, we have
 \begin{align*}
  (((U(\xi))^k-\xi)\widehat{v}(\xi))'=-\widehat{v}(\xi) <0,\\
  (((U_n(\xi))^k-\xi)\widehat{v}_n(\xi))'=-\widehat{v}_n(\xi) <0 \qquad \mbox{for }n=1,2,\dots.
 \end{align*}
Then, $\widehat{v}(U^k-\xi)$ and $\widehat{v}_n(U_n^k-\xi)$, $n=1, 2,\dots$, are monotone decreasing and continuous functions. Because the sequence of monotone functions 
which converges to a continuous function must converge uniformly, we get that $q_n(\xi)$ converges to zero uniformly. Then, from \eqref{PRegB}, \eqref{eqSeq1} and \eqref{eqSeq2}  it follows that
$$ \widehat{u}_n \to \widehat{u} \ \mbox{in }C^2([-R,R]), \mbox{ as } n \to \infty. $$
Therefore $T:K \to C^2([-R,R])$ is continuous. In addition, from \eqref{Sol1u}, we have
$$ \widehat{u}'(\xi)=\frac{(u_+-u_-)\exp \left( \int_{-R}^\xi \frac{\widehat{v}(U,s)((U(s))^k-s)}{\varepsilon} ds \right)}{\int_{-R}^R \exp \left( \int_{-R}^r \frac{\widehat{v}(U,s)((U(s))^k-s)}{\varepsilon} ds \right)dr} $$
which implies that $\widehat{u}=TU$ is monotone. So we get $TK \subset K$.
\end{proof}

\begin{lemma}
 $TK$ is a bounded set in $C^2([-R,R])$.
\end{lemma}
\begin{proof}
For any $U \in K$, if $s<\xi_\sigma^\varepsilon$, we have
\begin{equation} \label{eqEstim1}
 0<\widehat{v}(U,s)((U(s))^k-s)=\widehat{v}_-(u_-^k+R)-\int_{-R}^s \widehat{v}(r)dr <\widehat{v}_-(u_-^k+R)
\end{equation}
and if $s>\xi_\sigma^\varepsilon$,
\begin{equation} \label{eqEstim2}
 0>\widehat{v}(U,s)((U(s))^k-s)=\widehat{v}_+(u_+^k-R)+\int_{s}^R \widehat{v}(r)dr >\widehat{v}_+(u_+^k-R).
\end{equation}
From \eqref{PRegA}, we can deduce that
$$ \widehat{u}''(\xi) <0, \quad \xi \in [-R,\xi_\sigma^\varepsilon). $$
Then, $\widehat{u}'(\xi)\le \widehat{u}'(-R)<0$, $\xi \in [-R,\xi_\sigma^\varepsilon)$, and
\begin{align*}
 u_--u_+>\widehat{u}(-R)-\widehat{u}(\xi_\sigma)=\widehat{u}'(\zeta)(-R-\xi_\sigma)>\widehat{u}'(\zeta)(-R-u_+^k), \qquad \zeta \in (-R, \xi_\sigma).
\end{align*}
Thus,
$$ 0>\widehat{u}'(-R)>\widehat{u}'(\zeta)>-\frac{u_--u_+}{R+u_+^k}. $$
Also, from \eqref{PRegA} we have
$$ \widehat{u}'(\xi)=\widehat{u}'(-R)\exp \left(\int_{-R}^\xi \frac{\widehat{v}(U^k-s)}{\varepsilon}ds \right) $$
and by \eqref{eqEstim1} and \eqref{eqEstim2}, we conclude that $\widehat{u}'$ is uniformly bounded.
Consequently, $\widehat{u}''$ is also uniformly bounded. So, $TK$ is a bounded set in $C^2([-R,R])$.
\end{proof}

\begin{lemma}
 $TK$ is precompact in $C([-R,R])$.
\end{lemma}
\begin{proof}
 This is a consequence of the compact embedding $C^2([-R,R]) \hookrightarrow C([-R,R])$.
\end{proof}

From the above lemmas, by virtue of Schauder fixed point theorem, we
get the following result.
\begin{theorem} \label{Thm3.5}
 For each $R > \max \{ |u_-|^k,|u_+|^k \}$, there exists a weak solution $$ (\widehat{v}_R,\widehat{u}_R) \in L^1([-R, R]) \times C^2([-R, R]) $$
for the system \eqref{sysVis2} with boundary value $(\widehat{v}_R(\pm R), \widehat{u}_R(\pm R))=(v_\pm, u_\pm)$, and, in addition, being $\widehat{u}_R$ a decreasing function.
\end{theorem}

The next step is to obtain from this family of solutions a sequence $R_k \to \infty$ such that $(\widehat{v}_{R_k}, \widehat{u}_{R_k})$ converges to a weak solution of \eqref{sysVis2}--\eqref{datsysVisc2}.
To this end, we need the following lemma.
\begin{lemma} \label{lemma3.6}
\begin{enumerate}
\item $\widehat{u}_R(\xi)$, $\widehat{u}'_R(\xi)$ and $\widehat{u}''_R(\xi)$ are uniformly bounded, with respect to $R$ and $\xi\in [-R,R]$.
\item There exist a sequence $R_k\to\infty$ and a decreasing function $\widehat{u} \in C^1(\mathbb{R})$ such that $\widehat{u}_{R_k}$ converges to $\widehat{u}$ in $C^1([-M,M])$, for each positive number $M$ (i.e. $\widehat{u}_{R_k}$, $\widehat{u}'_{R_k}$ converge uniformly in compact sets of $\mathbb{R}$ to $\widehat{u}$, $\widehat{u}'$, respectively). 
\item $\widehat{v}_{R_k}(\widehat{u}_{R_k},\xi)$ converges to $\widehat{v}(\widehat{u},\xi)$, as $R_k\to\infty$, for each $\xi \in \mathbb{R} \setminus \{\xi_\sigma^\varepsilon \}$, where $\widehat{v}_{R_k}(\widehat{u}_{R_k},\xi)$, $\widehat{v}(\widehat{u},\xi)$ are defined accordingly with \eqref{solweak1}, \eqref{solRho1} and \eqref{solRho2}, being $R=\infty$ for $\rho(u,\xi)$, and $\xi_{\sigma}^\varepsilon$ satisfies $(\widehat{u}(\xi_\sigma^\varepsilon))^k=\xi_\sigma^\varepsilon$. 
\end{enumerate}
\end{lemma}
\begin{proof}
 \begin{enumerate}
  \item To simplify the notation in this proof, we shall use $\widehat{u}$, $\widehat{u}'$ and $\widehat{u}''$ instead of $\widehat{u}_R$, $\widehat{u}'_R$ and $\widehat{u}''_R$.\\
  Observe that $u_+^k < \xi_\sigma^\varepsilon<u_-^k$. We choose $\xi_1$ such that $-R<\xi_1<u_+^k$.
  From  \eqref{sysVis2} it follows that
  \begin{equation*}
   \widehat{u}'(\xi)=\widehat{u}'(\xi_1)\exp \left( \int_{\xi_1}^\xi \frac{\widehat{v}(s)((\widehat{u}(s))^k-s)}{\varepsilon}ds \right).
  \end{equation*}
As $\widehat{u}''(\xi)<0$ for $\xi \in (-R,\xi_\sigma^\varepsilon)$, then we have that $\widehat{u}'(\xi) <\widehat{u}'(\xi_1)<0$, $\xi \in (\xi_1,\xi_\sigma^\varepsilon)$.
Since
$$ u_--u_+ >\widehat{u}(\xi_1)-\widehat{u}(\xi_\sigma^\varepsilon)=\widehat{u}'(\zeta)(\xi_1-\xi_\sigma^\varepsilon)>\widehat{u}'(\zeta)(\xi_1-u_+^k), $$
where $\zeta\in (\xi_1,\xi_\sigma^\varepsilon)$, we get
$$ \widehat{u}'(\zeta)>\frac{u_--u_+}{\xi_1-u_+^k}, \qquad \zeta\in (\xi_1,\xi_\sigma^\varepsilon).$$
It follows that
$$ 0>\widehat{u}'(\xi_1)>\frac{u_--u_+}{\xi_1-u_+^k}. $$
When $\xi<\xi_1$, $$ \exp \left( \int_{\xi_1}^\xi \frac{\widehat{v}(s)((\widehat{u}(s))^k-s)}{\varepsilon}ds \right)< 1. $$
When $\xi_1<\xi<\xi_\sigma^\varepsilon$, observe that
\begin{align*}
 \widehat{v}_1(\xi_1)&=\widehat{v}_-\frac{u_-^k+R}{(\widehat{u}(\xi_1))^k-\xi_1}\exp \left(-\int_{-R}^{\xi_1} \frac{ds}{(\widehat{u}(s))^k-s} \right)\\
 &\le \widehat{v}_-\frac{u_-^k+R}{(\widehat{u}(\xi_1))^k-\xi_1}\exp \left(-\int_{-R}^{\xi_1} \frac{ds}{u_-^k-s} \right)
 =\widehat{v}_-\frac{u_-^k-\xi_1}{(\widehat{u}(\xi_1))^k-\xi_1}
\end{align*}
and
\begin{align*}
 \widehat{v}(\xi)((\widehat{u}(\xi))^k -\xi)&=\widehat{v}(\xi_1)((\widehat{u}(\xi_1))^k -\xi_1)-\int_{\xi_1}^\xi \widehat{v}(s) ds\\
 &\le \widehat{v}(\xi_1)((\widehat{u}(\xi_1))^k -\xi_1) \le v_-(u_-^k-\xi_1),
\end{align*}
and we obtain
$$ \exp \left( \int_{\xi_1}^\xi \frac{\widehat{v}(s)((\widehat{u}(s))^k-s)}{\varepsilon}ds \right) \le \exp \left( \frac{(v_-(u_-^k-\xi_1)^2}{\varepsilon} \right). $$ 
When $\xi>\xi_\sigma^\varepsilon$, we have
\begin{small}
\begin{align*}
 \int_{\xi_1}^\xi \frac{\widehat{v}(s)((\widehat{u}(s))^k-s)}{\varepsilon}ds &= \int_{\xi_1}^{\xi_\sigma^\varepsilon} \frac{\widehat{v}(s)((\widehat{u}(s))^k-s)}{\varepsilon}ds+\int_{\xi_\sigma^\varepsilon}^\xi \frac{\widehat{v}(s)((\widehat{u}(s))^k-s)}{\varepsilon}ds\\
 &<\int_{\xi_1}^{\xi_\sigma^\varepsilon} \frac{\widehat{v}(s)((\widehat{u}(s))^k-s)}{\varepsilon}ds
\end{align*}
\end{small}
or 
$$ \exp \left( \int_{\xi_1}^\xi \frac{\widehat{v}(s)((\widehat{u}(s))^k-s)}{\varepsilon}ds \right) < \exp \left( \int_{\xi_1}^{\xi_\sigma^\varepsilon} \frac{\widehat{v}(s)((\widehat{u}(s))^k-s)}{\varepsilon}ds \right). $$
Therefore, $\widehat{u}'(\xi)$ and $\widehat{u}(\xi)$ are uniformly bounded.
From \eqref{PRegA} it follows that $\widehat{u}''_R(\xi)$ is also uniformly bounded, with respect to $R$ and $\xi\in [-R,R]$.
\item Fixing $M>0$, we consider $R>>M$ and apply the Arzel{\'a}-Ascoli theorem to obtain a sequence $(\widehat{u}_{R_k})$ converging  in $C^1([-M,M])$ to a decreasing function $u$. Then, by a diagonalization  process, we obtain a sequence $(\widehat{u}_{R_k})$, which we do not relabel, such that $(\widehat{u}_{R_k})$ converges to a decreasing function $\widehat{u}\in C^1(\mathbb{R})$, uniformly in compact sets in $\mathbb{R}$,  $(\widehat{u}'_{R_k})$ also converges uniformly in compact sets in $\mathbb{R}$ to $\widehat{u}'$, and $\widehat{u}(-\infty)=u_{-}$, $\widehat{u}(\infty)=u_+$.
\item Claim 3 is obtained from \eqref{solRho1}, \eqref{solRho2} by passing to the limit as $R_k\to\infty$, for each fixed $\xi\not=\xi_\sigma^\varepsilon$,  noticing that, up to a subsequence, we can assume that $\xi_\sigma^{R_k}$, defined by $(u_{R_k}^k(\xi_\sigma^{R_k}))=\xi_\sigma^{R_k}$, converges to $\xi_\sigma^\varepsilon$ (where $\xi_\sigma^\varepsilon$ is defined by $(u(\xi_\sigma^\varepsilon))^k=\xi_\sigma^\varepsilon$).
 \end{enumerate}
\end{proof}

\begin{theorem} \label{Thm3.7}
 Let $\widehat{u}$ be the function obtained in Lemma \ref{lemma3.6}. Then, for each $\varepsilon>0$, $\widehat{u}$ satisfies
 \begin{equation*}
  \begin{cases}
   \varepsilon \widehat{u}''=(\widehat{v}(\widehat{u},\xi)(\widehat{u}^k-\xi))\widehat{u}',\\
   \widehat{u}(\pm \infty)=u_\pm,
  \end{cases}
 \end{equation*}
and
\begin{equation*}
 \widehat{v}(\xi)=\begin{cases}
            \widehat{v}_1(\xi), &\mbox{if }-\infty <\xi<\xi_\sigma^\varepsilon,\\
            \widehat{v}_2(\xi), &\mbox{if }\xi_\sigma^\varepsilon <\xi<\infty,\\
           \end{cases}
\end{equation*}
where $\xi_\sigma^\varepsilon$ satisfies $(u(\xi_\sigma^\varepsilon))^k = \xi_\sigma^\varepsilon$,
\begin{align*}
 \widehat{v}_1(\xi)=v_-\exp \left( -\int_{-\infty}^\xi \frac{((u(s))^k)'}{(u(s))^k-s} ds \right)
 \end{align*}
 and 
 \begin{align*}
  \widehat{v}_2(\xi)=v_+\exp \left( \int_{\xi}^{+\infty} \frac{((u(s))^k)'}{(u(s))^k-s} ds \right).
\end{align*}
\end{theorem}
\begin{proof}
Denote by $(\widehat{v}_R(\xi),\widehat{u}_R(\xi))$ the solution of the problem \eqref{sysVis2} with boundary value $(\widehat{v}(\pm R),\widehat{u}(\pm R)) = (v_\pm,u_\pm)$. Fixing $\xi_2$ and integrating \eqref{PRegA} from $\xi_2$ to $\xi$, we obtain
\begin{align*}
 \varepsilon (\widehat{u}'_R(\xi)-\widehat{u}'_R(\xi_2)) =& (\widehat{v}_R(\xi)((\widehat{u}_R(\xi))^k-\xi))\widehat{u}_R(\xi) \\
 &-(\widehat{v}_R(\xi_2)((\widehat{u}_R(\xi_2))^k-\xi_2))\widehat{u}_R(\xi_2)+\int_{\xi_2}^\xi \widehat{v}_R(s)\widehat{u}_R(s) ds
\end{align*}
(independently of whether $\xi_\sigma^\varepsilon$ is between $\xi_2$ and $\xi$).
Letting $R \to +\infty$, by the Lebesgue Convergence Theorem it follows that
\begin{align}
 \varepsilon (\widehat{u}'(\xi)-\widehat{u}'(\xi_2))=& (\widehat{v}(\xi)((\widehat{u}(\xi))^k-\xi))\widehat{u}(\xi) \nonumber \\
 &-(\widehat{v}(\xi_2)((\widehat{u}(\xi_2))^k-\xi_2))\widehat{u}(\xi_2) 
  +\int_{\xi_2}^\xi \widehat{v}(s)\widehat{u}(s) ds. \label{eqDiff}
\end{align}
Differentiating \eqref{eqDiff} with respect to $\xi$, 
we obtain
$$ \varepsilon \widehat{u}''=(\widehat{v}(\widehat{u}^k-\xi))\widehat{u}', $$ and from \eqref{Sol1u} we have $\widehat{u}(\pm \infty)=u_\pm$.
\end{proof}

\begin{theorem}
 There exists a weak solution $(\widehat{v},\widehat{u}) \in L^1_{loc}((-\infty,+\infty)) \times C^2((-\infty,+\infty))$ for the boundary value problem \eqref{sysVis2}--\eqref{datsysVisc2}.
\end{theorem}
\begin{proof}
Let $(\widehat{v}, \widehat{u})$ be defined in Theorem \ref{Thm3.7}. By Lemma \ref{lemma3.6} we know that $\widehat{u}$ is decreasing and 
belongs to $C^1(\mathbb{R})$. Then $\widehat{v}$ is of class $C^1$ in $(-\infty, \xi_\sigma^\varepsilon)\cup (\xi_\sigma^\varepsilon,\infty)$. In addition, it is also bounded, hence, locally integrable. From \eqref{eqDiff} it follows that $\widehat{u}$ is of class $C^2(\mathbb{R})$. The first equation in \eqref{sysVis2} is obtained by differentiating $\widehat{v}_1$ and $\widehat{v}_2$, and the second one is equivalent to the first one and the equation stated in Theorem \ref{Thm3.7}. 
\end{proof}

\section{The limit solutions of \texorpdfstring{\eqref{sysVis1}--\eqref{datsysVisc1}}{(5)--(6)} as viscosity vanishes}
We continue this section studying the case when $u_->u_+$ and we are interested in 
the behavior of the solutions $(\widehat{v}^\varepsilon,\widehat{u}^\varepsilon)$ of \eqref{sysVis2}--\eqref{datsysVisc2} as $\varepsilon \to 0+$.
\begin{lemma} \label{lemmaA}
 Let $\xi_\sigma^\varepsilon$ be the unique point satisfying $(\widehat{u}^\varepsilon(\xi_\sigma^\varepsilon))^k=\xi_\sigma^\varepsilon$, and let $\xi_\sigma$ be the limit $$ \xi_\sigma=\lim_{\varepsilon \to 0+} \xi_\sigma^\varepsilon $$
 (passing to a subsequence if necessary). Then for any $\eta >0$,
 \begin{align*}
  &\lim_{\varepsilon \to 0+} \widehat{u}_\xi^\varepsilon(\xi)=0, \qquad \mbox{for } |\xi-\xi_\sigma| \ge \eta,\\
  &\lim_{\varepsilon \to 0+} \widehat{u}^\varepsilon(\xi)=\begin{cases}
                                                 u_-, &\mbox{if } \xi \le \xi_\sigma -\eta,\\
                                                 u_+, &\mbox{if } \xi \ge \xi_\sigma +\eta,\\
                                                \end{cases}
 \end{align*}
uniformly in the above intervals.
\end{lemma}
\begin{proof} To simplify the notation in this proof, we use $\widehat{v}$, $\widehat{u}$ instead of $\widehat{v}^\varepsilon$, $\widehat{u}^\varepsilon$.

 Take $\xi_3=\xi_\sigma+\eta/2$, and let $\varepsilon$ be so small such that $\xi_\sigma^\varepsilon <\xi_3-\eta/4$. 
 For $\xi>\xi_\sigma$,
 \begin{align*}
  \widehat{v}(\xi)&=v_+ \exp \left( \int_\xi^{+\infty} \frac{((\widehat{u}(s))^k)'}{(\widehat{u}(s))^k-s} ds \right) \\
  &=\lim_{R\to +\infty} v_+ \frac{R-u_+^k}{\xi-(\widehat{u}(\xi))^k}\exp \left( \int_\xi^{R} \frac{ds}{(\widehat{u}(s))^k-s} \right)\\
  &\le \lim_{R\to +\infty} v_+ \frac{R-u_+^k}{\xi-(\widehat{u}(\xi))^k}\exp \left( \int_\xi^{R} \frac{ds}{u_+^k-s} \right) \\
  &=\lim_{R\to +\infty} v_+ \frac{R-u_+^k}{\xi-(\widehat{u}(\xi))^k}\frac{\xi-u_+^k}{R-u_+^k}
= v_+ \frac{\xi-u_+^k}{\xi-(\widehat{u}(\xi))^k},
 \end{align*}
 and we have
 \begin{equation*}
  \widehat{v}(\xi)((\widehat{u}(\xi))^k-\xi) \ge v_+(u_+^k-\xi), \qquad \qquad \xi\in (\xi_\sigma,+\infty).
 \end{equation*}
Now, integrating the second equation of \eqref{sysVis2} twice on $[\xi_3,\xi]$, we get
 \begin{align*}
  \widehat{u}(\xi_3)-\widehat{u}(\xi)&=-\widehat{u}'(\xi_3) \int_{\xi_3}^\xi \exp \left( \int_{\xi_3}^r \frac{\widehat{v}(s)((\widehat{u}(s))^k-s)}{\varepsilon}ds \right) dr\\
  &\ge -\widehat{u}'(\xi_3) \int_{\xi_3}^\xi \exp \left( \int_{\xi_3}^r \frac{v_+(u_+^k-s)}{\varepsilon} ds \right) dr\\
  &= -\widehat{u}'(\xi_3) \int_{\xi_3}^\xi \exp \left( \frac{v_+}{\varepsilon} \left( \left(u_+^k-\xi_3 \right)(r-\xi_3)-\frac12(r-\xi_3)^2 \right) \right) dr\\
  &= -\widehat{u}'(\xi_3) \int_{0}^{\xi-\xi_3} \exp \left( \frac{v_+}{\varepsilon} \left( \left(u_+^k-\xi_3 \right)r-\frac12 r^2 \right) \right) dr.
 \end{align*}
Letting $\xi \to +\infty$, we get
\begin{align*}
 u_--u_+ &\ge -\widehat{u}'(\xi_3) \int_{0}^{+\infty} \exp \left( \frac{v_+}{\varepsilon} \left( \left(u_+^k-\xi_3 \right)r-\frac12 r^2 \right) \right) dr\\
 &\ge -\widehat{u}'(\xi_3) \int_{0}^{2\varepsilon} \exp \left( \frac{v_+}{\varepsilon} \left( \left(u_+^k-\xi_3 \right)r-\frac12 r^2 \right) \right) dr\\
 &\ge -\widehat{u}'(\xi_3) \sqrt{\varepsilon}A_3
\end{align*}
for $0\le \varepsilon \le 1$, where $A_3$ is a constant independent of $\varepsilon$. Thus
\begin{equation*}
 |\widehat{u}'(\xi_3)| \le \frac{u_--u_+}{\sqrt{\varepsilon}A_3}.
\end{equation*}
So
\begin{equation} \label{eqC}
 |\widehat{u}'(\xi)| \le \frac{u_--u_+}{\sqrt{\varepsilon}A_3}\exp \left( \int_{\xi_3}^\xi \frac{\widehat{v}(s)((\widehat{u}(s))^k-s)}{\varepsilon}ds \right).
\end{equation}
For $\xi>\xi_3$,
\begin{align*}
 \widehat{v}(\xi)&=\lim_{R \to +\infty} v_+ \frac{R-u_+^k}{\xi-(\widehat{u}(\xi))^k}\exp \left( \int_\xi^R \frac{ds}{(\widehat{u}(s))^k-s} \right)\\
 &\ge \lim_{R \to +\infty} v_+ \frac{R-u_+^k}{\xi-(\widehat{u}(\xi))^k}\exp \left( \int_\xi^R \frac{ds}{(\widehat{u}(\xi_3))^k-s} \right)\\
 &=v_+\frac{\xi-(\widehat{u}(\xi_3))^k}{\xi-(\widehat{u}(\xi))^k} \lim_{R \to +\infty} \frac{R-u_+^k}{R-(\widehat{u}(\xi_3))^k}
 =v_+\frac{\xi-(\widehat{u}(\xi_3))^k}{\xi-(\widehat{u}(\xi))^k}
\end{align*}
and we have
\begin{equation} \label{eqD}
 \widehat{v}(\xi)((\widehat{u}(\xi))^k-\xi) \le v_+((\widehat{u}(\xi_3))^k-\xi), \qquad \qquad \xi >\xi_3.
\end{equation}
From \eqref{eqC} and \eqref{eqD} we have
\begin{equation*}
 |\widehat{u}'(\xi)| \le \frac{u_--u_+}{\sqrt{\varepsilon}A_3}\exp \left( -\frac{v_+}{\varepsilon} \int_{\xi_3}^\xi \left( s- (\widehat{u}(\xi_3))^k \right) ds \right)
\end{equation*}
which implies that
$$ \lim_{\varepsilon \to 0+} \widehat{u}_\xi^\varepsilon(\xi)=0, \qquad \mbox{uniformly for }\xi \ge \xi_\sigma+\eta. $$
Now, we choose $\xi$ and $\xi_4$ such that $\xi>\xi_4 \ge \xi_\sigma+\eta$.
From the relation
$$ \widehat{u}(\xi_4)-\widehat{u}(\xi)=-\widehat{u}'(\xi_4)\int_{\xi_4}^\xi \exp \left( \int_{\xi_4}^r
\frac{\widehat{v}(s)((\widehat{u}(s))^k-s)}{\varepsilon} ds \right) dr, $$
we get
\begin{align*}
 | \widehat{u}(\xi_4)-\widehat{u}(\xi)| &\le |\widehat{u}'(\xi_4)| \int_{\xi_4}^\xi \exp \left(
-\frac{A_4}{\varepsilon}(r-\xi_4) \right)dr \\
&\le \frac{\varepsilon}{A_4}|\widehat{u}'(\xi_4)| \left(1-\exp \left(
\frac{A_4}{\varepsilon}(\xi_4-\xi) \right) \right),
\end{align*}
where $A_4=v_+ \left( \xi_4- (\widehat{u}(\xi_4))^k \right)$.
When $\xi \to+\infty$, we obtain
\begin{equation*}
 |\widehat{u}(\xi_4)-u_+| \le \frac{\varepsilon}{A_4}|\widehat{u}'(\xi_4)|,
\end{equation*}
which implies that
\begin{equation*}
 \lim_{\varepsilon \to 0+} \widehat{u}^\varepsilon(\xi)=u_+, \qquad \mbox{uniformly
for } \xi \ge \xi_\sigma+\eta.
\end{equation*}
The results for $\xi <\xi_\sigma-\eta$ can be obtained analogously.
\end{proof}

\begin{lemma} \label{lemmaB}
 For any $\eta>0$,
 \begin{equation*}
  \lim_{\varepsilon \to 0+}\widehat{v}^\varepsilon(\xi)=\begin{cases}
                             v_-, &\mbox{if } \xi<\xi_\sigma-\eta,\\
                             v_+, &\mbox{if } \xi>\xi_\sigma+\eta,
                            \end{cases}
 \end{equation*}
uniformly, with respect to $\xi$.
\end{lemma}
\begin{proof}
Take $\varepsilon_0>0$ so small such that $|\xi_\sigma^\varepsilon-\xi_\sigma| < \frac{\eta}{2}$ whenever $0<\varepsilon<\varepsilon_0$. For any
$\xi > \xi_\sigma+\eta$ and $\varepsilon <\varepsilon_0$, we have $$ \xi >\xi_\sigma^\varepsilon+\frac{\eta}{2} $$
and
$$ \widehat{v}^\varepsilon(\xi)=v_+ \exp \left( \int_{\xi}^\infty \frac{((\widehat{u}^\varepsilon(s))^k)'}{(\widehat{u}^\varepsilon(s))^k-s}ds \right). $$
For any $s\in [\xi,+\infty)$, we have
\begin{align*}
 (\widehat{u}^\varepsilon(s))^k-s &< (\widehat{u}^\varepsilon(\xi))^k-\xi=(1-((\widehat{u}^\varepsilon(\zeta))^k)')(\xi_\sigma^\varepsilon-\xi)
\le -\frac{\eta}{2}.
\end{align*}
As $\widehat{u}$ is decreasing, we have that $((\widehat{u}(\xi))^k)'=k(\widehat{u}(\xi))^{k-1}\widehat{u}'(\xi)<0$, and
$$ \frac{((\widehat{u}(s))^k)'}{(\widehat{u}^\varepsilon(s))^k-s} <-\frac{2}{\eta}((\widehat{u}(s))^k)' , \qquad \mbox{for any } s \in [\xi,+\infty). $$
Now, in the last inequality, integrating on $[\xi,+\infty)$ we have
$$ 0\le \int_{\xi}^\infty \frac{((\widehat{u}^\varepsilon(s))^k)'}{(\widehat{u}^\varepsilon(s))^k-s} ds \le -\frac{2}{\eta}\int_{\xi}^\infty((u^\varepsilon(s))^k)' ds= -\frac{2}{\eta}(u_+^k-(\widehat{u}^\varepsilon(\xi))^k), $$
so
\begin{equation} \label{inq1}
1 \le \exp \left( \int_{\xi}^\infty \frac{((\widehat{u}^\varepsilon(s))^k)'}{(\widehat{u}^\varepsilon(s))^k-s} ds \right) \le \exp \left( -\frac{2}{\eta}(u_+^k-(\widehat{u}^\varepsilon(\xi))^k) \right).
\end{equation}
By Lemma~\ref{lemmaA} we have that $ \lim\limits_{\varepsilon \to 0+} \widehat{u}^\varepsilon(\xi)=u_+$, and from \eqref{inq1} we have 
$$ \lim_{\varepsilon \to 0+} \exp \left( \int_{\xi}^\infty \frac{((\widehat{u}^\varepsilon(s))^k)'}{(\widehat{u}^\varepsilon(s))^k-s} ds \right)=1 $$ and
$$ \lim_{\varepsilon \to 0+} \widehat{v}^\varepsilon(\xi)= \lim_{\varepsilon \to 0+} v_+ \exp \left( \int_{\xi}^\infty \frac{((\widehat{u}^\varepsilon(s))^k)'}{(\widehat{u}^\varepsilon(s))^k-s} ds \right)=v_+, \ \mbox{uniformly for } \xi>\xi_\sigma+\eta. $$
Similarly, we also obtain that $\lim\limits_{\varepsilon\to0}\widehat{v}^\varepsilon(\xi)=v_{-}$, uniformly for $\xi<\xi_\sigma-\eta$.
\end{proof}
Now, we study the limit behavior of $(\widehat{v}^\varepsilon, \widehat{u}^\varepsilon)$ in the
neighborhood of $\xi_\sigma$ as $\varepsilon \to 0+$.

\begin{theorem}
 Denote
\begin{equation} \label{limSigma}
 \sigma=\xi_\sigma=\lim_{\varepsilon \to 0+}
\xi_\sigma^\varepsilon=\lim_{\varepsilon \to 0+}
(\widehat{u}^\varepsilon(\xi_\sigma^\varepsilon))^k=(\widehat{u}(\sigma))^k.
\end{equation}
  Then
  \begin{equation*}
   \lim_{\varepsilon \to 0+} (\widehat{v}^\varepsilon(\xi), \widehat{u}^\varepsilon(\xi))=\begin{cases}
                                                                          (v_-,u_-), &\mbox{if }\xi <\sigma,\\
                                                                          (w_0\cdot \delta, u_\delta), &\mbox{if }\xi =\sigma,\\
                                                                          (v_+,u_+), &\mbox{if }\xi >\sigma,
                                                                         \end{cases}
  \end{equation*}
where $\widehat{v}^\varepsilon(\xi)$ converges in the sense of the distributions to the sum of a step function and a Dirac measure $\delta$ with weight $w_0=-\sigma(v_--v_+)+(v_-u_-^k-v_+u_+^k)$.
 \end{theorem}
 \begin{proof}
As $\sigma=\xi_\sigma=\lim\limits_{\varepsilon \to 0+}
(\widehat{u}^\varepsilon(\xi_\sigma^\varepsilon))^k=(\widehat{u}(\sigma))^k$, then we have
\begin{equation} \label{entropyCond}
 u_+^k<\sigma <u_-^k.
\end{equation}
Let $\xi_1$ and $\xi_2$ be real numbers such that $\xi_1<\sigma<\xi_2$ and $\phi\in C_0^\infty([\xi_1,\xi_2])$ such that
$\phi(\xi)\equiv \phi(\sigma)$ for $\xi$ in a
neighborhood $\Omega$ of $\sigma$, $\Omega \subset (\xi_1,\xi_2)$ \footnote{The function $\phi$ is called a {\em sloping test function} \cite{TZZ}}. Then $\xi_\sigma^\varepsilon \in \Omega$ whenever $0<\varepsilon<\varepsilon_0$. 
From \eqref{sysVis2} we have
\begin{equation} \label{eqL1}
-\int_{\xi_1}^{\xi_2} \widehat{v}^\varepsilon((\widehat{u}^\varepsilon)^k-\xi)\phi' d\xi+\int_{\xi_1}^{\xi_2} \widehat{v}^\varepsilon \phi d\xi=0.
\end{equation}
For $\alpha_1, \alpha_2 \in \Omega$, $\alpha_1,\alpha_2$ near $\sigma$ such that $\alpha_1 <\sigma <\alpha_2$, we write
\begin{equation*}
 \int_{\xi_1}^{\xi_2} \widehat{v}^\varepsilon((\widehat{u}^\varepsilon)^k-\xi)\phi' d\xi=\int_{\xi_1}^{\alpha_1} \widehat{v}^\varepsilon((\widehat{u}^\varepsilon)^k-\xi)\phi' d\xi+\int_{\alpha_2}^{\xi_2} \widehat{v}^\varepsilon((\widehat{u}^\varepsilon)^k-\xi)\phi' d\xi,
\end{equation*}
and from Lemmas~\ref{lemmaA} and \ref{lemmaB}, we obtain
\begin{align*}
 \lim_{\varepsilon \to 0+} \int_{\xi_1}^{\xi_2} \widehat{v}^\varepsilon((\widehat{u}^\varepsilon)^k-\xi)\phi' d\xi =&
 \int_{\xi_1}^{\alpha_1} v_-(u_-^k-\xi)\phi' d\xi+\int_{\alpha_2}^{\xi_2} \widehat{v}_+(u_+^k-\xi)\phi' d\xi\\
 =& \left( v_-u_-^k-v_+u_+^k-v_-\alpha_1+v_+\alpha_2 \right) \phi(\sigma)\\
 &+\int_{\xi_1}^{\alpha_1}v_-\phi(\xi)d\xi
 +\int_{\alpha_2}^{\xi_2}v_+\phi(\xi)d\xi
\end{align*}
Then taking $\alpha_1 \to \sigma-$, $\alpha_2 \to \sigma+$, we arrive at
\begin{equation} \label{eqL2}
 \lim_{\varepsilon \to 0+} \int_{\xi_1}^{\xi_2} \widehat{v}^\varepsilon((\widehat{u}^\varepsilon)^k-\xi)\phi' d\xi=
 \left( -[\widehat{v}]\sigma+[\widehat{v} \widehat{u}^k] \right) \phi(\sigma)+\int_{\xi_1}^{\xi_2} J(\xi-\sigma)\phi(\xi)d\xi
\end{equation}
where $[q]=q_--q_+$ and $$J(x)=\begin{cases}
                               v_-, &\mbox{if }x<0,\\
                               v_+, &\mbox{if }x>0.
                              \end{cases}$$
From \eqref{eqL1} and \eqref{eqL2}, we get
\begin{equation*}
 \lim_{\varepsilon \to 0+} \int_{\xi_1}^{\xi_2} (\widehat{v}^\varepsilon-J(\xi-\sigma))\phi(\xi)d\xi=\left( -[\widehat{v}]\sigma+[\widehat{v} \widehat{u}^k] \right) \phi(\sigma).
\end{equation*}
for all sloping test functions $\phi \in C_0^\infty ([\xi_1,\xi_2])$.\\
For an arbitrary $\psi \in C_0^\infty([\xi_1,\xi_2])$, we take a sloping test function $\phi$, such that $\phi(\sigma)=\psi(\sigma)$ and $$ \max_{[\xi_1,\xi_2]}|\psi-\phi|<\mu, $$ for a sufficiently small $\mu>0$. As $\widehat{v}^\varepsilon \in L^1([\xi_1,\xi_2)$ uniformly, we obtain
\begin{align*}
 \lim_{\varepsilon \to 0+} \int_{\xi_1}^{\xi_2} (\widehat{v}^\varepsilon-J(\xi-\sigma))\psi(\xi)d\xi &= 
 \lim_{\varepsilon \to 0+} \int_{\xi_1}^{\xi_2} (\widehat{v}^\varepsilon-J(\xi-\sigma))\phi(\xi)d\xi+O(\mu)\\
 &=\left( -[\widehat{v}]\sigma+[\widehat{v} \widehat{u}^k] \right) \phi(\sigma)+O(\mu)\\
 &=\left( -[\widehat{v}]\sigma+[\widehat{v} \widehat{u}^k] \right) \psi(\sigma)+O(\mu).
\end{align*}
Then, when $\mu \to 0+$, we find that 
\begin{equation} \label{deltaStr}
 \lim_{\varepsilon \to 0+} \int_{\xi_1}^{\xi_2} (\widehat{v}^\varepsilon-J(\xi-\sigma))\psi(\xi)d\xi=\left( -[\widehat{v}]\sigma+[\widehat{v} \widehat{u}^k] \right) \psi(\sigma)
\end{equation}
holds for all test functions $\psi \in C_0^\infty ([\xi_1,\xi_2])$. Thus, $\widehat{v}^\varepsilon$ converges in the sense of the distributions to the sum of a step function and a Dirac delta function with strength $-[\widehat{v}]\sigma+[\widehat{v} \widehat{u}^k]$.
In a similar way, from
\begin{equation} \label{eqR1}
 -\int_{\xi_1}^{\xi_2} \left( \widehat{v}^\varepsilon ((\widehat{u}^\varepsilon)^k-\xi) \right) \widehat{u}^\varepsilon \phi' d\xi+\int_{\xi_1}^{\xi_2} \widehat{v}^\varepsilon \widehat{u}^\varepsilon \phi d\xi = \varepsilon \int_{\xi_1}^{\xi_2} (\widehat{u}^\varepsilon)'' \phi d\xi
\end{equation}
we can obtain
\begin{equation*}
 \lim_{\varepsilon \to 0+} \int_{\xi_1}^{\xi_2} (\widehat{v}^\varepsilon \widehat{u}^\varepsilon-\widetilde{J}(\xi-\sigma))\phi(\xi)d\xi=\left( -[\widehat{v} \widehat{u}]\sigma+\left[ \widehat{v} \widehat{u}^{k+1} \right] \right) \phi(\sigma)
\end{equation*}
for all test functions $\phi \in C_0^\infty ([\xi_1,\xi_2])$, where 
$$ \widetilde{J}(x)=\begin{cases}
                               v_-u_-, &\mbox{if }x<0,\\
                               v_+u_+, &\mbox{if }x>0.
                              \end{cases}$$
Thus $\widehat{v} \widehat{u}$ also converges in the sense of the distributions to the sum of a step function and a Dirac delta function with strength $-[\widehat{v} \widehat{u}]\sigma+\left[ \widehat{v} \widehat{u}^{k+1} \right]$.\\
 If we take the test function in \eqref{eqR1} as $\frac{\psi}{\widetilde{u}^\varepsilon+\nu}$, $\nu >0$, where $\widetilde{u}^\varepsilon$ is a modified function satisfying $\widehat{u}^\varepsilon(\sigma)$ in $\Omega$ and $\widehat{u}^\varepsilon$ outside $\Omega$, and let $\nu \to 0+$, we find
 \begin{equation}\label{deltaStr2}
  \lim_{\varepsilon \to 0+} \int_{\xi_1}^{\xi_2} (\widehat{v}^\varepsilon -J(\xi-\sigma))\psi d\xi \cdot \widehat{u}(\sigma) = \left( -[\widehat{v} \widehat{u}]\sigma+\left[\widehat{v} \widehat{u}^{k+1} \right] \right) \psi(\sigma)
 \end{equation}
for all test functions $\psi \in C_0^\infty ([\xi_1,\xi_2])$.
Let $w_0$ be the strength of the Dirac delta function in $\widehat{v}$, and denote
$$ \widehat{u}_\delta=\lim_{\varepsilon \to 0+} \widehat{u}^\varepsilon(\xi_\sigma^\varepsilon)=\widehat{u}(\sigma). $$
From \eqref{limSigma}, \eqref{deltaStr} and \eqref{deltaStr2} it follows that
\begin{equation} \label{GRHsys2}
 \begin{cases}
  \sigma = (\widehat{u}_\delta)^k,\\
  w_0=-\sigma[\widehat{v}]+[\widehat{v} \widehat{u}^k],\\
  w_0 u_\delta = -\sigma[\widehat{v} \widehat{u}]+\left[ \widehat{v} \widehat{u}^{k+1} \right].
 \end{cases}
\end{equation}
Under the entropy condition \eqref{entropyCond} the system \eqref{GRHsys2} admits a unique solution $(\sigma, w_0,u_\delta)$.
\end{proof}
Then we get the following theorem.
\begin{theorem} \label{ThmFinal}
 Suppose $u_->u_+$. Let $(\widehat{v}^\varepsilon(x,t),\widehat{u}^\varepsilon(x,t))$ be the similarity solution of \eqref{sysVis1}--\eqref{datsysVisc1}. Then the limit $$ \lim_{\varepsilon \to 0+}(\widehat{v}^\varepsilon(x,t),\widehat{u}^\varepsilon(x,t))=(\widehat{v}(x,t),\widehat{u}(x,t)) $$
 exists in the measure sense and $(\widehat{v},\widehat{u})$ solves \eqref{sys1}--\eqref{datsysVisc1}.
 Moreover, 
 \begin{small}
 \begin{equation*}
  (\widehat{v}(x,t),\widehat{u}(x,t))=\begin{cases}
                      (v_-,u_-), &\mbox{if } x<\frac{\sigma}{\alpha k} (1-e^{-\alpha kt}),\\
                      (\frac{w_0}{\alpha k}(1-e^{-\alpha kt})\delta(x-\frac{\sigma}{\alpha k} (1-e^{-\alpha kt})), u_\delta), &\mbox{if } x=\frac{\sigma}{\alpha k} (1-e^{-\alpha kt}),\\
                      (v_+,u_+), &\mbox{if } x>\frac{\sigma}{\alpha k} (1-e^{-\alpha kt}),
                     \end{cases}
 \end{equation*} 
 \end{small}
where the constants $\sigma$, $w_0$ , and $u_\delta$ are determined uniquely by the entropy condition $u_+^k <\sigma <u_-^k$ and
\begin{equation*}
 \begin{cases}
  \sigma =  u_\delta^k,\\
  w_0=-\sigma[\widehat{v}]+[\widehat{v} \widehat{u}^k],\\
  w_0 u_\delta = -\sigma[\widehat{v} \widehat{u}]+\left[ \widehat{v} \widehat{u}^{k+1} \right].
 \end{cases}
\end{equation*}
\end{theorem}

\section{Delta shock solutions for the system \texorpdfstring{\eqref{system_ld}}{(1)}}
In this section, we study the Riemann problem to the system \eqref{system_ld} with initial data \eqref{datoRiemann} when $u_- > u_+$. We need recall the following definition:
\begin{definition}
 A two-dimensional weighted delta function $w(s)\delta_L$ supported on a smooth curve $L=\{ (x(s),t(s))\, :\, a < s < b \}$, for $w\in L^1((a,b))$, is defined as
 \begin{equation*}
  \langle w(\cdot)\delta_L,\phi(\cdot,\cdot) \rangle = \int_a^b w(s) \phi(x(s),t(s)) \, ds, \quad \phi \in C_0^\infty(\mathbb{R} \times[0,\infty)). 
 \end{equation*}
\end{definition}
Now, we define a delta shock wave solution for the system \eqref{system_ld} with initial data \eqref{datoRiemann}.
\begin{definition}
 A distribution pair  $(v,u)$ is a {\em delta shock wave solution} of \eqref{system_ld} and \eqref{datoRiemann} in the sense of distribution if there exists a smooth curve $L$ and a function $w \in C^1(L)$ such that $v$ and $u$ are represented in the following form
 $$ v = \widetilde{v}(x,t)+w \delta_L \text{ and } u=\widetilde{u}(x,t), $$
 $\widetilde{v}, \widetilde{u} \in L^\infty(\mathbb{R} \times(0,\infty); \mathbb{R})$ and
 \begin{equation} \label{weakSol1}
  \begin{cases}
   \langle v, \varphi_t \rangle + \langle v u^k, \varphi_x \rangle =0,\\
   \langle v u, \varphi_t \rangle + \langle v u^{k+1} , \varphi_x \rangle = \langle \alpha vu, \varphi \rangle,
  \end{cases}
 \end{equation}
for all the test functions $\varphi \in C_0^\infty (\mathbb{R} \times(0,\infty))$, where $u|_L=u_\delta(t)$ and
\begin{align*}
 \langle v,\varphi \rangle &= \int_0^\infty \int_\mathbb{R} \widetilde{v} \varphi \, dx dt + \langle w \delta_L, \varphi \rangle,\\
 \langle v G(u),\varphi \rangle &= \int_0^\infty \int_\mathbb{R} \widetilde{v} G(\widetilde{u}) \varphi \, dx dt + \langle w G(u_\delta) \delta_L, \varphi \rangle.
\end{align*}
\end{definition}
With the previous definitions, we are going to find a solution with discontinuity $x=x(t)$ for \eqref{system_ld} of the form
\begin{equation} \label{deltaSol1}
 (v(x,t),u(x,t))=\begin{cases}
                          (v_-(x,t),u_-(x,t)), &\text{if } x<x(t),\\
                          (w(t)\delta_L,u_\delta(t)), &\text{if } x=x(t),\\
                          (v_+(x,t),u_+(x,t)), &\text{if } x>x(t),
                      \end{cases}
\end{equation}
where $v_\pm(x,t)$, $u_\pm(x,t)$ are piecewise smooth solutions of system \eqref{system_ld}, $\delta(\cdot)$ is the 
 Dirac measure supported on the curve $x(t) \in C^1$, and $x(t)$, $w(t)$ and $u_\delta(t)$ will be determined later.\\
 
Since $v(x,t)=\widehat{v}(x,t)$ and $u(x,t)=\widehat{u}(x,t)e^{-\alpha t}$, from Theorem \ref{ThmFinal}, we can establish a solution of the form \eqref{deltaSol1} to the system \eqref{system_ld} with initial data \eqref{datoRiemann}. Thus, we have the following result.
\begin{theorem} \label{Thm4.1}
 Assume that $u_->u_+$.
 Then the Riemann problem \eqref{system_ld}--\eqref{datoRiemann} admits one and only one measure solution of the
form
\begin{equation} \label{sol_final}
  (v(x,t),u(x,t))=\begin{cases}
                      (v_-,u_-e^{-\alpha t}), &\mbox{if } x<x(t),\\
                      (w(t) \delta(x-x(t)), u_\delta e^{-\alpha t}), &\mbox{if } x=x(t),\\
                      (v_+,u_+ e^{-\alpha t}), &\mbox{if } x>x(t),
                     \end{cases}
 \end{equation}
where $w(t)=\frac{w_0}{\alpha k}(1-e^{-\alpha kt})$, $x(t)=\frac{\sigma}{\alpha k} (1-e^{-\alpha kt})$ and the constants $\sigma$, $w_0$ , and $u_\delta$ are determined uniquely by the entropy condition $u_+^k <\sigma < u_-^k$ and
\begin{equation} \label{sysEqEDM-O}
 \begin{cases}
  \sigma = u_\delta^k,\\
  w_0=-\sigma(v_--v_+)+(v_- u_-^k-v_+ u_+^k),\\
  w_0 u_\delta = -\sigma(v_- u_--v_+ u_+)+ \left( v_- (u_-)^{k+1} -v_+ (u_+)^{k+1} \right).
 \end{cases}
\end{equation}
\end{theorem}
\begin{proof}
We need to show that \eqref{sol_final} is a solution to the problem \eqref{system_ld}--\eqref{datoRiemann} which can be found with $(v,u)=(\widehat{v},\widehat{u}e^{-\alpha t})$ and the result obtained in Theorem \ref{ThmFinal}. 
Therefore, for any test function $\varphi \in C_0^\infty (\mathbb{R} \times (0,\infty))$ we have
\begin{small}
\begin{align*}
\langle v u, \varphi_t \rangle + \langle v u^{k+1} , \varphi_x \rangle =& \int_0^\infty \int_\mathbb{R} (vu \varphi_t+vu^{k+1} \varphi_x) dx dt \\
&+ \int_0^\infty w(t) u_\delta e^{-\alpha t} (\varphi_t+ u_\delta^{k} e^{-\alpha kt}\varphi_x) dt \\
=&\int_0^\infty \int_{-\infty}^{x(t)} (v_-u_-e^{-\alpha t} \varphi_t+v_-u_-^{k+1}e^{-\alpha (k+1)t} \varphi_x) dx dt \\
&+ \int_0^\infty \int_{x(t)}^{\infty} (v_+u_+e^{-\alpha t} \varphi_t+v_+u_+^{k+1}e^{-\alpha (k+1)t} \varphi_x) dx dt\\
&+\int_0^\infty w(t) u_\delta e^{-\alpha t} (\varphi_t+ u_\delta^{k} e^{-\alpha kt} \varphi_x) dt \\
=&-\oint - \left( v_-u_-^{k+1}e^{-\alpha (k+1)t} \varphi \right) dt + \left( v_-u_-e^{-\alpha t} \varphi \right) dx\\
&+ \oint - \left( v_+u_+^{k+1}e^{-\alpha (k+1)t} \varphi \right) dt + \left( v_+u_+e^{-\alpha t} \varphi \right) dx\\
&+ \int_0^\infty \int_{\mathbb{R}} \alpha vu \varphi dx dt
+\int_0^\infty w(t) u_\delta e^{-\alpha t} (\varphi_t+ \sigma e^{-\alpha kt} \varphi_x) dt \\
=&\int_0^\infty (v_-u_-^{k+1}-v_+u_+^{k+1})e^{-\alpha (k+1)t} \varphi dt \\ 
& - \int_0^\infty \frac{dx(t)}{dt} (v_-u_--v_+u_+)e^{-\alpha t}  \varphi dt\\
&+ \int_0^\infty \int_{\mathbb{R}} \alpha vu \varphi dx dt
+\int_0^\infty w(t) u_\delta e^{-\alpha t} (\varphi_t+ \frac{dx(t)}{dt} \varphi_x) dt \\
=&\int_0^\infty  (v_-u_-^{k+1}-v_+u_+^{k+1})e^{-\alpha (k+1)t} \varphi dt \\
& - \int_0^\infty \frac{dx(t)}{dt} (v_-u_--v_+u_+)e^{-\alpha t}  \varphi dt\\
&-\int_0^\infty \frac{w_0 u_\delta}{\alpha k} \frac{d ( (1-e^{-\alpha kt}) e^{-\alpha t})}{dt} \varphi dt
+ \int_0^\infty \int_{\mathbb{R}} \alpha vu \varphi dx dt \\
=&\int_0^\infty \int_{\mathbb{R}} \alpha vu \varphi dx dt + \int_0^\infty \alpha w(t) u_\delta e^{-\alpha t} \varphi dt
=\langle \alpha vu, \varphi \rangle
\end{align*}
\end{small}
which implies the second equation of \eqref{weakSol1}. By a completely similar argument it is possible to obtain the first equation of \eqref{weakSol1}.
\end{proof}

As an application of Theorem \ref{Thm4.1} we have the following result:
\begin{corollary}[Keita and Bourgault \cite{KB}]
Consider the following Eulerian droplet model
\begin{equation*}
\begin{cases}
v_t+(vu)_x=0,\\
(vu)_t+(vu^2)_x=-\alpha uv,
\end{cases}
\end{equation*}
with initial data given by
\begin{equation*}
(v(x,0),u(x,0))=
\begin{cases}
(v_-,u_-), &\mbox{if } x<0,\\
(v_+,u_+), &\mbox{if } x>0,
\end{cases}
\end{equation*}
and assume that $v_\pm >0$ and $u_->u_+$. Then the Riemann solution to the 
Eulerian droplet model is given by
\begin{equation*}
(v(x,t),u(x,t))=
\begin{cases}
(v_-,u_-e^{-\alpha t}), &\mbox{if } x<x(t),\\
(w(t)\delta(x-x(t)),u_\delta e^{-\alpha t}), &\mbox{if } x=x(t),\\
(v_+,u_+e^{-\alpha t}), &\mbox{if } x>x(t),
\end{cases}
\end{equation*}
where $w(t)=\sqrt{v_-v_+}(u_--u_+) \frac{(1-e^{-\alpha t})}{\alpha}$, $x(t)=\frac{\sqrt{v_-}u_-+\sqrt{v_+}u_+}{\sqrt{v_-}+\sqrt{v_+}} \frac{(1-e^{-\alpha t})}{\alpha}$ and $u_\delta=\frac{\sqrt{v_-}u_-+\sqrt{v_+}u_+}{\sqrt{v_-}+\sqrt{v_+}}$, when $v_- \neq v_+$.\\
For the case $v_-=v_+$, $w(t)=v_-(u_--u_+) \frac{(1-e^{-\alpha t})}{\alpha}$, $x(t)= \frac{(u_-+u_+)(1-e^{-\alpha t})}{2 \alpha}$ and $u_\delta=\frac{1}{2}(u_-+u_+)$.
\end{corollary}
\begin{proof}
Using the Theorem \ref{Thm4.1} with $k=1$, we need solve the following system 
\begin{equation} \label{sysEqEDM}
 \begin{cases}
  \sigma = u_\delta,\\
  w_0=-\sigma(v_--v_+)+(v_- u_--v_+ u_+),\\
  w_0 u_\delta = -\sigma(v_- u_--v_+ u_+)+ \left( v_- u_-^2 -v_+ u_+^2 \right),
 \end{cases}
\end{equation}
subject to $u_+ < \sigma < u_-$. Thus, when $v_- \neq v_+$, from the system \eqref{sysEqEDM} we have
$$ (v_--v_+)u_\delta^2 -2(v_-u_--v_+u_+) u_\delta +(v_-u_-^2-v_+u_+^2)=0 $$
and therefore we can find
$$ u_\delta^{(1)} = \frac{\sqrt{v_-}u_--\sqrt{v_+}u_+}{\sqrt{v_-}-\sqrt{v_+}}
\qquad \mbox{ and } \qquad
 u_\delta^{(2)} = \frac{\sqrt{v_-}u_-+\sqrt{v_+}u_+}{\sqrt{v_-}+\sqrt{v_+}}.$$
 Observe that only $u_\delta^{(2)}$ satisfies the condition $u_+ < \sigma < u_-$.
 Thus, with $u_\delta = u_\delta^{(2)}$, from \eqref{sysEqEDM} we have that
 $$ w_0 = \sqrt{v_-v_+}(u_--u_+). $$
 When $v_-=v_+$, from \eqref{sysEqEDM} we have
 $$ 2(u_--u_+) u_\delta -(u_-^2-u_+^2)=0 $$
 and therefore
 $$ u_\delta = \frac{1}{2}(u_-+u_+) \qquad \mbox{ and } \qquad w_0= v_-(u_--u_+). $$
\end{proof}

\section{Final remarks} 
\begin{enumerate}
\item From Theorem \ref{Thm4.1}, we can observe that when $\alpha \to 0+$, 
$$ \lim_{\alpha \to 0+} x(t) =u_\delta^k t \qquad \mbox{ and } \qquad \lim_{\alpha \to 0+} w(t)=-u_\delta^k (v_--v_+)t+(v_-u_-^k-v_+u_+^k)t. $$
So, the limit behavior of the solution given in Theorem \ref{Thm4.1} as $\alpha \to 0+$ corresponds to the solution given in Theorem 5.6 of \cite{HYang}.
\item It is easy to see that \eqref{deltaSol1} is a delta shock solution with discontinuity $x=x(t)$ to the problem \eqref{system_ld}--\eqref{datoRiemann} if and only if the following generalized Rankine-Hugoniot conditions are satisfied
\begin{equation} \label{gR-H}
\begin{cases}
\frac{dx(t)}{dt}= (u_\delta(t))^k=\sigma(t),\\
\frac{dw(t)}{dt}= -\llbracket v \rrbracket \sigma(t)+\llbracket vu^k \rrbracket,\\
\frac{d(w(t)u_\delta(t))}{dt}=-\llbracket vu \rrbracket \sigma(t)+\llbracket vu^{k+1} \rrbracket-\alpha w(t)u_\delta(t),
\end{cases}
\end{equation}
where $\llbracket q \rrbracket=q(x(t)-,t)-q(x(t)+,t)$. Observe that the solution given in Theorem \ref{Thm4.1} satisfies the generalized Rankine-Hugoniot conditions. Moreover, in this case the generalized Rankine-Hugoniot conditions \eqref{gR-H} are equivalent to the equations given in the system \eqref{sysEqEDM-O}.
\end{enumerate}

{\bf Conflict of interest:} This work does not have any conflicts of interest.

\end{document}